\documentclass[bjps]{imsart}

\RequirePackage[OT1]{fontenc}
\usepackage{amsthm,amsmath,natbib,amsfonts,amssymb}
\RequirePackage{hyperref}

\startlocaldefs
\numberwithin{equation}{section}
\theoremstyle{plain}
\newtheorem{thm}{Theorem}[section]
\newtheorem{defi}{Definition}[section]
\endlocaldefs

\begin{document}

\begin{frontmatter}
\title{Nonparametric Tests for Bivariate Stochastic Dominance without Continuity Assumptions}
\runtitle{Tests for Bivariate Stochastic Dominance}

\begin{aug}

\author{\fnms{Luciano Alejo} \snm{Perez}\thanksref{a}%
\ead[label=e3]{luperez@itba.edu.ar}}%


\runauthor{Luciano Alejo Perez}

\affiliation[a]{Instituto Tecnol\'ogico de Buenos Aires}

\address{Departamento de Matem\'atica,\\
Escuela de Ingenier\'ia y Gesti\'on,\\
Instituto Tecnol\'ogico de Buenos Aires,\\
\printead{e3}}

\end{aug}

\begin{abstract}

The use of Kolmogorov-Smirnov-type statistics for testing stochastic dominance goes back to \cite{mcfadden}. In this paper we extend the approach of \cite{barret1} to the bivariate case, without the assumption of absolute continuity for the underlying distributions. Using empirical processes theory and bootstrap techniques we obtain consistent nonparametric tests for bivariate first and second order stochastic dominance, over several modularity classes of test functions. This tests are in turn useful tools in applied fields such as multidimensional economic inequality as shown by \cite{perez1}.

\end{abstract}




\end{frontmatter}

\section{Introduction}

Since its introduction by \cite{hadarrussell69}, stochastic dominance has been an important tool both in risk theory and economics. The work of \cite{atkinson70} introduced the concept of stochastic dominance in the context of income inequality, where it has  played ever since a prominent role, given by the welfare interpretation of stochastic dominance based on Bergson-Samuelson functionals. 

In general terms, stochastic dominance is a mathematical rule for ordering random prospects based on their expected values over certain classes of test functions. For example, given random variables $X$ and $Y$ with cdfs $F_{X}$ and $F_{Y}$ respectively, with common compact support $\bar{U}$, and a class $\Phi$ of real valued functions, we say there is stochastic dominance of $X$ over $Y$ for the class $\Phi$, if for every $\phi\in\Phi$ we have:

\begin{equation}
 \int_{\mathbb{R}}\phi(t)dF_{X}\geq\int_{\mathbb{R}}\phi(t)dF_{Y}
\end{equation}

The most important classes of functions in the univariate case are the class $\Phi_{1}$ of $\mathcal{C}^{1}$ increasing functions and $\Phi_{2}$ of $\mathcal{C}^{2}$ increasing and concave functions. If $X$ dominates $Y$ over $\Phi_{1}$ we say that there is first order stochastic dominance, and we denote it by $XSD_{1}Y$. If $X$ dominates $Y$ over $\Phi_{2}$ we say that there is second order stochastic dominance, denoted $XSD_{2}Y$. 

\cite{hadarrussell69} provided sufficient conditions for first and second order dominance. If we assume (without loss of generality) that $F_{X}$ and $F_{Y}$ have support $\bar{U}=[0,1]$ we have:

\begin{eqnarray}
\label{suffSD1D1}F_{X}(s)\leq F_{Y}(s) & \forall s\in[0,1]\Rightarrow & XSD_{1}Y\\
\label{suffSD1D2}\int_{0}^{t}F_{X}(s)ds\leq \int_{0}^{t}F_{X}(s)ds & \forall t\in[0,1]\Rightarrow & XSD_{2}Y 
\end{eqnarray}

As conditions \ref{suffSD1D1} and \ref{suffSD1D2} are obviously also necessary, we have equivalence of those conditions and stochastic dominance. The importance of this conditions lies on the fact that they involve a single inequality for cdfs (or its integrals) as opposed to the definition of stochastic dominance which involves an infinite set of inequalities (one for each test function $\phi$).

This in turn allows for an easy definition of univariate stochastic dominance of arbitrary order using the following sequence of integral operators:

\begin{equation}\label{jthordSD}
\mathcal{T}_{1}(z,F_{X})=F_{X}(z) \hspace{1cm} \mathcal{T}_{j+1}(z,F_{X})=\int_{0}^{z}\mathcal{T}_{j}(t,F_{X})dt 
\end{equation}

We say that there is \textit{j-th order stochastic dominance} of $X$ over $Y$ if $\mathcal{T}_{j}(z,F_{X})\leq\mathcal{T}_{j}(z,F_{Y})$. 

Additionally, and perhaps more important in practice, conditions \ref{suffSD1D1} and \ref{suffSD1D2} involve only cdfs as inputs, which makes easier to test them statistically. 

The idea of using Kolmogorov-Smirnov-type statistics to test for univariate stochastic dominance of $X$ over $Y$ comes from \cite{mcfadden} who stated the null hypothesis $H_{0}: F_{X}(z)\leq F_{Y}(z)$ $\forall z\in [0,1]$ against alternative $H_{1}: F_{X}(z)>F_{Y}(z)$ for some $z\in [0,1]$.  

As dominance condition $F_{X}(z)\leq F_{Y}(z)$ $\forall z\in [0,1]$ is equivalent to $\sup_{z\in[0,1]}{(F_{X}-F_{Y})(z)}\leq0$ it is natural to search for an empirical or \textit{plug-in} analogue of this inequality.

Using same size samples $(x_{1},...,x_{n})$ and $(y_{1},...,y_{n})$ to construct empirical cdfs $F_{X}^{n}$ and $F_{Y}^{n}$ we can define $D_{n}(z)=\sqrt{n}(F_{X}^{n}(z)-F_{Y}^{n}(z))$ and then a two sample test statistic for $H_{0}$ is:

\begin{equation}\label{KSSD1D1}
D_{n}^{*}=\max_{w\in [0,1]}D_{n}(w) 
\end{equation}

The only information we need then to carry on testing $XSD_{1}Y$ in the univariate case is the asymptotic distribution of test statistic $D_{n}^{*}$, which \cite{mcfadden} shows to be of Smirnov type (assuming continuity for $F_{X}$ and $F_{Y}$). The \textit{pivotality} of Kolmogorov-Smirnov statistics such as \ref{KSSD1D1}, that is, its independence on the underlying distributions, is the key factor here, as those distributions are generally unknown. 

As proven by \cite{simpson} this invariance property for the asymptotic distribution of Kolmogorov-Smirnov statistics is lost in the two dimensional case.  On the other hand \cite{schmid} shows how discontinuity accounts for loss of pivotality even for the univariate case. To work around these difficulties we'll have to use empirical processes theory and bootstrap techniques in Sections \ref{SD1} and \ref{SD2}. 

For second order univariate stochastic dominance \cite{mcfadden} proposes the test statistic: 

\begin{equation}\label{KSSD2D1}
S_{n}^{*}=\max_{w\in [0,1]}S_{n}(w) 
\end{equation}

where $S_{n}(w)=\sqrt{n}\int_{0}^{w}(F_{Y}^{n}(z)-F_{X}^{n}(z))dz$, but didn't completely characterize its asymptotic distribution.

\cite{barret1} revisited the use of Kolmogorov-Smirnov-type statistics for testing j-th order stochastic dominance. In the same way as \cite{mcfadden}, they start with empirical or plug-in versions of the integral operators \ref{jthordSD} and to test the null hypothesis of stochastic dominance, they used general two-sample test statistics, allowing for different sample sizes. To characterize limit distributions of their test statistics they used both parametric bootstrap simulations and empirical bootstrap, obtaining \textit{data dependent p-values}. 

Our goal in Sections \ref{SD1} and \ref{SD2} is to develop bivariate first and second order stochastic dominance tests for distributions with compact support\footnote{Compact support for underlying distributions is not essential in first and second order dominance. For higher dominance orders, however, integral operators as in \ref{jthordSD} are in general unbounded and need to be dealt with using additional tools such as weighted Kolmogorov-Smirnov statistics as \cite{horvath}.} in some sense extending \cite{barret1} univariate tests, but using empirical bootstrap theory as in Section 3.7 of \cite{vanderV}. 

\section{Modularity Classes and Bivariate Stochastic Dominance}\label{classes}

Given a class $\Phi=\{\phi\}$ of continuous functions $\phi:\mathbb{R}^{2}\rightarrow\mathbb{R}$ we say that there is stochastic dominance of random vector $(X_{1},Y_{1})$ over $(X_{2},Y_{2})$ if for all $\phi\in\Phi$ we have $E(\phi(X_{1},Y_{1}))\geq E(\phi(X_{2},Y_{2}))$. 

If we assume that $(X_{1},Y_{1})$ and $(X_{2},Y_{2})$ have a cdfs $F_{1}$ and $F_{2}$ respectively, with common compact support $U$, we can put dominance in terms of integrals: 

\begin{equation}
 \int_{U}\phi(s,t)dF_{1}\geq\int_{U}\phi(s,t)dF_{2}\hspace{0.3cm}\forall\phi\in\Phi
\end{equation}

As in the univariate case, to obtain sufficient conditions for dominance in terms of inequalities involving only cdfs and it's integrals, we have to restrict our attention to some special classes of test functions. 

Following \cite{atkinson82} and \cite{perez2} we define first order modularity classes\footnote{Modularity conditions can be stated more generally without differentiability assumptions and conditions \ref{mod1} shown to be equivalent for the case of $\mathcal{C}^{2}$ functions. See \cite{topkis} or \cite{perez1}}.  

\begin{defi}\label{mod1}
Let $\Omega\subseteq\mathbb{R}^{2}$ be an open connected set and let $\phi:\mathbb{R}^{2}\rightarrow\mathbb{R}\in\mathcal{C}^{2}(\Omega)$ be an increasing function on each of its arguments. We say that:  
\begin{enumerate}
 \item $\phi\in\Phi^{+}$ if $\forall(x,y)\in\Omega$ we have $\frac{\partial^{2}\phi}{\partial x\partial y}(x,y)\geq0$.
 \item $\phi\in\Phi^{-}$ if $\forall(x,y)\in\Omega$ we have $\frac{\partial^{2}\phi}{\partial x\partial y}(x,y)\leq0$.
\end{enumerate}
In the first case we say that $\phi$ is ``supermodular'' while in the second case we call it ``submodular''. \hfill$\square$
\end{defi}

In the context of microeconomics, if we think of $\phi$ as a utility function , the condition of supermodularity accounts for \textit{strategic complementarity} of input goods (see \cite{topkis}). This is what makes bivariate stochastic dominance a useful tool in welfare economics as shown by \cite{perez1}. 

For any cdf $F$ with support $\bar{U}=[0,1]\times[0,1]$, we compute its marginal distributions $F^{X}(x)=F(x,1)$ and $F^{Y}(y)=F(1,y)$. Following the notation of \cite{atkinson82} we then define the associated function $K(x,y)=-(F(x,y)-F^{X}(x)-F^{Y}(y))$.
 
Given random vectors $(X_{1},Y_{1})$ and $(X_{2},Y_{2})$ with cdfs $F_{1}$ and $F_{2}$ and $(s,t)\in\bar{U}$ we denote $\Delta F(s,t)=F_{1}(s,t)-F_{2}(s,t)$,  $\Delta F^{X}(s)=F_{1}^{X}(s)-F_{2}^{X}(s)$, $\Delta F^{Y}(t)=F_{1}^{Y}(t)-F_{2}^{Y}(t)$  and $\Delta K(s,t)=K_{1}(s,t)-K_{2}(s,t)$. 

\cite{hadarrussell74} provided sufficient conditions for stochastic dominance of vector $(X_{1},Y_{1})$ over $(X_{2},Y_{2})$ over modularity classes $\Phi^{+}$ and $\Phi^{-}$ assuming absolutely continuous distributions. \cite{perez2} proved that those conditions are sufficient for the general case where underlying distributions are only assumed to have compact support.

This conditions are analogue to univariate sufficient conditions of \cite{hadarrussell69} in the sense that they involve just cdfs $F_{1}$ and $F_{2}$ and its integrals. For first order bivariate stochastic dominance we can state them as in \cite{perez1}:

\begin{eqnarray}\label{sd1cond}
        \Delta F(s,t)\leq0 \hspace{0.1cm}\forall (s,t)\in\bar{U} \Rightarrow F_{1}SD_{1}F_{2}\hspace{0.1cm}\text{over}\hspace{0.1cm}\Phi^{-}\\
        \Delta K(s,t)\leq0, \Delta F^{X}(s),\Delta F^{Y}(t)\leq0\hspace{0.1cm}\forall (s,t)\in\bar{U} \Rightarrow F_{1}SD_{1}F_{2}\hspace{0.1cm}\text{over}\hspace{0.1cm}\Phi^{+}
 \end{eqnarray}

Note that the specific marginal conditions are only relevant in the supermodular case, as they are trivial in the submodular case.

To obtain sufficient conditions for bivariate second order dominance we first have to set our higher order modularity classes. We follow \cite{perez2} and define: 

\begin{defi}\label{mod2}
Let $\Omega\subseteq\mathbb{R}^{2}$ be an open connected set and let $\phi:\mathbb{R}^{2}\rightarrow\mathbb{R}\in\mathcal{C}^{2}(\Omega)$ be an increasing function on each of its arguments. We say that:  
\begin{enumerate}
 \item $\phi\in\Phi^{--}$ if $\phi\in\Phi^{-}$ and $\forall(x,y)\in\Omega$ we have:
\begin{eqnarray}
  \frac{\partial^{2}\phi}{\partial x^{2}}(x,y),\frac{\partial^{2}\phi}{\partial y^{2}}(x,y)\leq0  \\ 
  \frac{\partial^{3}\phi}{\partial x^{2}\partial y}(x,y),\frac{\partial^{3}\phi}{\partial x \partial y^{2}}(x,y)\geq0 \\
  \frac{\partial^{4}\phi}{\partial x^{2}\partial y^{2}}(x,y)\leq0   
\end{eqnarray}
\item $\phi\in\Phi^{++}$ if $\phi\in\Phi^{+}$ and $\forall(x,y)\in\Omega$ we have:
\begin{eqnarray}
  \frac{\partial^{2}\phi}{\partial x^{2}}(x,y),\frac{\partial^{2}\phi}{\partial y^{2}}(x,y)\leq0  \\ 
  \frac{\partial^{3}\phi}{\partial x^{2}\partial y}(x,y),\frac{\partial^{3}\phi}{\partial x \partial y^{2}}(x,y)\leq0 \\
  \frac{\partial^{4}\phi}{\partial x^{2}\partial y^{2}}(x,y)\geq0   
\end{eqnarray}
\end{enumerate}
\hfill$\square$
\end{defi}

To establish again sufficient conditions involving a single inequality we have to to define integral functionals:

\begin{eqnarray}\label{intops}
H(x,y,F)=\int_{0}^{x}\int_{0}^{y}F(s,t)dtds\\
L(x,y,F)=\int_{0}^{x}\int_{0}^{y}K(s,t)dtds
\end{eqnarray}

Those functionals are defined for any cdf $F$ with support $\bar{U}$ and in particular with subindex 1 and 2 we denote its analogue for $F_{1}$ and $F_{2}$. 

We are also going to need marginal functionals defined as:

\begin{eqnarray}\label{intopsmargin}
H^{X}(x,F)=\int_{0}^{x}F^{X}(s)ds\\
H^{Y}(y,F)=\int_{0}^{y}F^{Y}(t)dt
\end{eqnarray}

We set difference operators as before: $\Delta H(s,t)=H_{1}(s,t)-H_{2}(s,t)$, $\Delta H^{X}(s,t)=H_{1}^{X}(s,t)-H_{2}^{X}(s,t)$, $\Delta H^{Y}(s,t)=H_{1}^{Y}(s,t)-H_{2}^{Y}(s,t)$ and $\Delta L(s,t)=L_{1}(s,t)-L_{2}(s,t)$. The following conditions were proven to be sufficient by \cite{hadarrussell74} for the absolutely continuous case and by \cite{perez2} for the general compact support case:

\begin{eqnarray}\label{sd2cond}
        \Delta H\leq0, \Delta H^{X}\leq0, \Delta H^{Y}\leq0\hspace{0.1cm}\forall (s,t)\in\bar{U} \Rightarrow  F_{1}SD_{2}F_{2}\hspace{0.1cm}\text{over}\hspace{0.1cm}\Phi^{--}\\
        \Delta L\leq0, \Delta H^{X}\leq0, \Delta H^{Y}\leq0\hspace{0.1cm}\forall (s,t)\in\bar{U} \Rightarrow  F_{1}SD_{2}F_{2}\hspace{0.1cm}\text{over}\hspace{0.1cm}\Phi^{++}
\end{eqnarray}

Note that marginal conditions are the same in both cases. 

Our goal in the next sections is to transform this sufficient conditions for dominance into statistical tests, using appropriate empirical versions of the functionals and Kolmogorov-Smirnov-type test statistics. 

\section{Testing for First Order Bivariate Stochastic Dominance}\label{SD1}

\subsection{General Setting}

To test for first order stochastic dominance we first extend the definition of empirical functionals from \cite{davidson} to the bivariate case. Given bivariate random samples $\{(X_{1},Y_{1}),...,(X_{1},Y_{1})\}$ we construct the empirical measure $P_{n}(B)=\frac{1}{n}\sum_{i=1}^{n}\delta_{(X_{i},Y_{i})}(B)$ for each Borel set $B\subseteq\mathbb{R}^{2}$. As vectors $(X_{i},Y_{i})$ are random, this is a random measure\footnote{On a technical note, we are going to assume that the sample space is a separable metric space and that continuum hypothesis holds, to ensure that the range of Brownian bridge is separable (see \cite{dudleyRAP}, Sec.13.1).}. 

The empirical distribution function is obtained as:

\begin{equation}\label{empcdf}
 \hat{F}_{n}(s,t)=P_{n}((-\infty,s]\times(-\infty,t])=\frac{1}{n}\sum_{i=1}^{n}1_{\{X_{i}\leq s\}}(s)1_{\{Y_{i}\leq t\}}(t)
\end{equation}

Marginal empirical cdfs are thus:

\begin{eqnarray}
 \hat{F}^{X}_{n}(s)=P_{n}((-\infty,s]\times\mathbb{R})=\frac{1}{n}\sum_{i=1}^{n}1_{\{X_{i}\leq s\}}(s)\\
 \hat{F}^{Y}_{n}(t)=P_{n}(\mathbb{R}\times(-\infty,t])=\frac{1}{n}\sum_{i=1}^{n}1_{\{Y_{i}\leq t\}}(t)
\end{eqnarray}

We also have an empirical version of $K(s,t)$ given by:

\begin{equation}\label{empiK}
\hat{K}(s,t)=-\left(\hat{F}(s,t)-\hat{F}^{X}(s)-\hat{F}^{Y}(t)\right)
\end{equation}

Now assume we have samples $\{(X_{1},Y_{1})_{1},(X_{1},Y_{1})_{2},...,(X_{1},Y_{1})_{m}\}$ from distribution $F_{1}$ (measure $P$) and $\{(X_{2},Y_{2})_{1},(X_{2},Y_{2})_{2},...,(X_{2},Y_{2})_{n}\}$ from distribution $F_{2}$ (measure $Q$), with associated empirical measures $P_{m}$ and $Q_{n}$ respectively. The general two sample statistic is $\nu_{m,n}=\left(\frac{mn}{m+n}\right)^{1/2}(P_{m}-Q_{n})$.

Following the same logic as in \cite{mcfadden} we are going to apply $\nu_{m,n}$ to sets of the form $[0,s]\times[0,t]$, and then take supremum to obtain a uniform condition of Kolmogorov-Smirnov type. 

We then have the test statistic for first order stochastic dominance in the submodular bivariate case:

\begin{equation}\label{lambda}
 \lambda_{m,n}=\left(\frac{mn}{m+n}\right)^{1/2}\sup_{(s,t)\in[0,1]\times[0,1]}(\hat{F_{1}}_{m}-\hat{F_{2}}_{n})(s,t)
\end{equation}

where $\hat{F_{1}}_{m}$ and $\hat{F_{2}}_{n}$ are the cdfs obtained from empirical measures $P_{m}$ and $Q_{n}$. 

For the supermodular class we introduce the analogue test statistic\footnote{Note that we are not centering our attention on the marginal conditions for dominance as they can be tested with the already known univariate test by \cite{barret1}.}:

\begin{equation}\label{kappa}
 \kappa_{m,n}=\left(\frac{mn}{m+n}\right)^{1/2}\sup_{[0,1]\times[0,1]}(\hat{K_{1}}_{m}-\hat{K_{2}}_{n})(s,t)
\end{equation}

To obtain limit distributions for test statistics \ref{lambda} and \ref{kappa} we are going to use bootstrap techniques as in \cite{vanderV}, Sec.3.7. To simplify notation put $V_{i}=(X_{1},Y_{1})_{i}$, the i-th element in the sample from $F_{1}$, and $W_{i}=(X_{2},Y_{2})_{j}$, the j-th element in the sample from $F_{2}$.

We start we the \textit{pooled sample} of size $N=n+m$:

\begin{equation}\label{pooledS}
 \{Z_{N1},Z_{N2},...,Z_{NN}\}=\{V_{1},...,V_{m},W_{1},...,W_{n}\}
\end{equation}

From the pooled sample we can define the \textit{pooled empirical distribution}:

\begin{equation}
 H_{N}=\frac{1}{N}\sum_{i=1}^{N}\delta_{Z_{Ni}}=\alpha_{N}P_{m}+(1-\alpha_{N})Q_{n}
\end{equation}

where $\alpha_{N}=m/n$. 

Now if we re-sample from \ref{pooledS} with replacement we obtain bootstrap samples of size $N$ and for each of this samples we have bootstrapped empirical measures: 

\begin{equation}\label{bootempi}
 \hat{P}_{m,N}^{B}=\frac{1}{m}\sum_{i=1}^{m}\delta_{\hat{Z}_{Ni}} \hspace{0.5cm} \hat{Q}_{n,N}^{B}=\frac{1}{n}\sum_{i=1}^{n}\delta_{\hat{Z}_{N,m+i}}
\end{equation}

We are going to use a couple of results from \cite{vanderV} but we first need to introduce some notation. Given a signed measure $Q$ and a class of measurable functions $\mathcal{F}$ we think of $Q$ acting on functions by integration, i.e. $Qf=\int fdQ$. We denote $\parallel Q\parallel_{\mathcal{F}}=\sup\{|Qf|:f\in\mathcal{F}\}$. We denote weak convergence of a sequence of probability measures by $\leadsto$.  

For the definition of Donsker classes and the Brownian bridge we refer to \cite{vanderV} Sec.2.1, and for outer integral $P^{*}f$ to Sec.1.2 in the same reference. 

We'll use the following results from Sec.3.7. of \cite{vanderV}:

\begin{thm}\label{boot1}
Let $\mathcal{F}$ be a class of measurable functions which is both $P$-Donsker and $Q$-Donsker, and such that  $||P||_{\mathcal{F}}<\infty$, $||Q||_{\mathcal{F}}<\infty$.  If $m,n\rightarrow\infty$ such that $m/N\rightarrow\alpha\in(0,1)$, then $\sqrt{m}(\hat{P}_{m,N}^{B}-H_{N})\leadsto G_{H}$ in probability given $V_{1}.V_{2},...W_{1},W_{2},...$. 

Here $G_{H}$ is a tight Brownian bridge corresponding to measure $H=\alpha P+(1-\alpha) Q$.
\end{thm}

\begin{proof}
It is Theorem 3.7.6 from \cite{vanderV}.\footnote{Theorem 3.7.6 and Theorem 3.7.7 are for general sample spaces, we are using them for random vectors in 2 dimensions.}
\end{proof}

\begin{thm}\label{boot2}
If in addition $\mathcal{F}$ possesses an envelope function $F$ with $P^{*}F^{2}<\infty$ and $Q^{*}F^{2}<\infty$, then $\sqrt{m}(\hat{P}_{m,N}-H_{N})\leadsto G_{H}$ given almost any sequence $V_{1}.V_{2},...W_{1},W_{2},...$. 
\end{thm}

\begin{proof}
It is Theorem 3.7.6 from \cite{vanderV}.
\end{proof}

\subsection{Nonparametric Tests for First Order Bivariate Stochastic Dominance}

In what follows we are going to consider the class of indicator functions: 

\begin{displaymath}
\mathcal{F}=\left\{1_{[0,s]\times[0,t]}:s,t\in[0,1]\right\} 
\end{displaymath}

and the space $\mathit{l}^{\infty}(\mathcal{F})$ of all real bounded functions over $\mathcal{F}$. A typical element from $\mathit{l}^{\infty}(\mathcal{F})$ is for example a finite signed measure, acting over the elements of $\mathcal{F}$ by integration. 

It will be useful to assign to each function $f\in\mathcal{F}$ a couple of related functions also in $\mathcal{F}$, given by:

\begin{displaymath}
 f=1_{[0,s]\times[0,t]}\longrightarrow \begin{cases}
                                        f^{x}=1_{[0,s]\times[0,1]} \\
                                        f^{y}=1_{[0,1]\times[0,t]}
                                       \end{cases}
\end{displaymath}

Let's consider now the following operators, $T_{1},T_{2}:\mathit{l}^{\infty}(\mathcal{F})\rightarrow\mathbb{R}$, given by:

\begin{equation}
T_{1}(P)=\sup_{f\in\mathcal{F}}Pf
\end{equation}

\begin{equation}
T_{2}(P)=\sup_{f\in\mathcal{F}}{(Pf^{x}+Pf^{y}-Pf)}
\end{equation}

Note that if $P$ y $Q$ are probability measures on $[0,1]\times[0,1]$, then they have associated cdfs $F_{1}(s,t)$ and $F_{2}(s,t)$ (and functions $K_{1}(s,t),K_{2}(s,t)$ given as in Sec.\ref{classes}) and then: 

\begin{equation}
 T_{1}(P-Q)=\sup_{(s,t)\in[0,1]\times[0,1]}(F_{1}-F_{2})(s,t)
\end{equation}

\begin{equation}
 T_{2}(P-Q)=\sup_{(s,t)\in[0,1]\times[0,1]}(K_{1}-K_{2})(s,t)
\end{equation}

\begin{thm}[First Order Bivariate Stochastic Dominance Test]\label{mioSD1}:\\ Let's consider test statistics $\lambda_{m,n}^{B}$ and $\kappa_{m,n}^{B}$ given by using two sample bootstrap empirical distributions (given by \ref{bootempi}) in formulas \ref{lambda} y \ref{kappa}. Then if $m,n\rightarrow\infty$ such that $m/N\rightarrow\alpha\in(0,1)$:

\begin{enumerate}
 \item $\lambda_{m,n}^{B}\leadsto T_{1}(G_{H})$ 
 \item $\kappa_{m,n}^{B}\leadsto T_{2}(G_{H})$
\end{enumerate}
given almost any sequence $V_{1},V_{2},...W_{1},W_{2},...$, where $H=\alpha P+(1-\alpha)Q$ y $G_{H}$ is a version of the tight Brownian bridge corresponding to measure $H$. 
\end{thm}

\begin{proof} Our class of functions $\mathcal{F}$ is a Donsker class because it is a subclass of rectangles in $\mathbb{R}^{2}$ which os known to be Donsker (\cite{dudleyRAP}, \cite{vanderV}). On the other hand, constant function 1 is an envelope satisfying conditions in Theorem \ref{boot2}. Consequently we have:

\begin{displaymath}
 \sqrt{\frac{mn}{m+n}}(\hat{P}_{m,N}-\hat{Q}_{n,N})\leadsto G_{H}
\end{displaymath}

Now we have $\lambda_{m,n}^{B}=T_{1}(\sqrt{\frac{mn}{m+n}}(\hat{P}_{m,N}-\hat{Q}_{n,N}))$ and $\kappa_{m,n}^{B}=T_{2}(\sqrt{\frac{mn}{m+n}}(\hat{P}_{m,N}-\hat{Q}_{n,N}))$. Note that empirical measures (and its bootstrap versions) are random elements taking values in the metric space $(\mathit{l}^{\infty}(\mathcal{F}),\parallel \parallel_{\mathcal{F}})$.

Under our assumptions we know that the range of the Brownian bridge is separable and then, as weak convergence is preserved by continuous mappings (\cite{vanderV} Theorem 1.11.1) our result will follow if we show that operators $T_{1}$ y $T_{2}$ are continuous.

Now if $P$ and $Q$ are elements from $\mathit{l}^{\infty}(\mathcal{F})$:

\begin{displaymath}
 |T_{1}(P)-T_{1}(Q)|=|\sup_{f\in\mathcal{F}}{Pf}-\sup_{f\in\mathcal{F}}{Qf}|\leq\sup_{f\in\mathcal{F}}|{(P-Q)f}|=\parallel P-Q\parallel_{\mathcal{F}}
\end{displaymath}

this shows that $T_{1}$ is continuous and proves the first part.

Now let $P$ and $Q$ be elements from $\mathit{l}^{\infty}(\mathcal{F})$, then:

\begin{multline*}
|T_{2}(P)-T_{2}(Q)|=|\sup_{f\in\mathcal{F}}{(Pf^{x}+Pf^{y}-Pf)}-\sup_{f\in\mathcal{F}}{(Qf^{x}+Qf^{y}-Qf)}|\\\leq\sup_{f\in\mathcal{F}}{\left\{|(Pf^{x}+Pf^{y}-Pf)-(Qf^{x}+Qf^{y}-Qf)|\right\}}\leq\\\leq\sup_{f\in\mathcal{F}}{\left\{|(P-Q)f^{x}|+(P-Q)f^{y}|+|(Q-P)f)|\right\}}
\end{multline*}

It is clear that $\sup_{f\in\mathcal{F}}{\left\{|(P-Q)f^{x}|\right\}}\leq\parallel P-Q\parallel_{\mathcal{F}}$, as we are covering a smaller class of functions. The same applies to the term with $\left\{|(P-Q)f^{y}|\right\}$ and then $|T_{2}(P)-T_{2}(Q)|\leq3\parallel P-Q\parallel_{\mathcal{F}}$, which shows $T_{2}$ is continuous and this finishes our proof.
\end{proof}

As a consequence of Theorem \ref{mioSD1}, we have a sequence of critical values given by the quantiles of bootstrap distribution, that is:

\begin{equation}
\hat{c}_{nm}^{-}=\inf{\left\{t:P_{\hat{Z}}(\lambda_{mn}^{B}>t)\leq\beta\right\}}
\end{equation}

for the submodular case and:

\begin{equation}
\hat{c}_{nm}^{+}=\inf{\left\{t:P_{\hat{Z}}(\kappa_{mn}^{B}>t)\leq\beta\right\}}
\end{equation}

for the supermodular case.

The consistency of the test then follows in the sense that:

\begin{eqnarray}
 \hat{c}_{nm}^{-}\longrightarrow c_{H}^{-}=\inf{\left\{t:P(T_{1}(G_{H})>t)\leq\beta\right\}}\\
 \hat{c}_{nm}^{+}\longrightarrow c_{H}^{+}=\inf{\left\{t:P(T_{2}(G_{H})>t)\leq\beta\right\}}
\end{eqnarray}

In the supermodular case for example, we have a one-tail test, rejecting hypothesis $PSD_{1}Q$ (equivalent to $K_{1}\leq K_{2}$ uniformly) if $\kappa_{m,n}>\hat{c}_{nm}^{+}$. Given the pooled sample $V_{1},...,V_{m},W_{1},...,W_{n}$, there are $N^{N}$ bootstrap samples, each with the same probability then $N^{N}$ possible values for $\kappa_{mn}^{B}$ (counting with multiplicity). 

Note that rejecting $PSD_{1}Q$ does not imply accepting $QSD_{1}P$, simply because $sup_{s,t}(F_{1}(s,t)-F_{2}(s,t))\nleq0$ does not imply $sup_{s,t}(F_{2}(s,t)-F_{1}(s,t))\leq0$. That is why we don't have a two tailed test, and instead we have to test hypothesis $PSD_{1}Q$ and $QSD_{1}P$ separately. This is similar for the tests in \cite{barret1}. 

Our test rejects the hypothesis if the value of $\kappa_{mn}$ is on the unpper $\beta$ fraction of the values $\kappa_{mn}^{B}$, resulting then in a probability of type I error less than $\beta$. The sequence of tests then results asymptotically of level $\beta$.

\section{Testing for Second Order Bivariate Stochastic Dominance}\label{SD2}

To test second order dominance we first need an empirical version of integral operator $H$:

\begin{equation}
 \hat{H}(x,y)=H(x,y,\hat{F}_{N})=\int_{0}^{x}\int_{0}^{y}\hat{F}_{N}(s,t)dsdt
\end{equation}

Noting that:

\begin{multline}
L(x,y)=\int_{0}^{x}\int_{0}^{y}K(s,t)dt=\int_{0}^{x}\int_{0}^{y}(F^{X}(s)+F^{Y}(t)-F(s,t))dsdt=\\=-H(x,y)+yF^{X}(x)+xF^{Y}(y) 
\end{multline}

we can simply use $\hat{H}$, $\hat{F}^{X}$ and $\hat{F}^{Y}$ to obtain $\hat{L}$. 

Using the same two samples as in Section \ref{SD1} we define for the submodular case:

\begin{equation}\label{mu}
 \mu_{m,n}=\left(\frac{mn}{m+n}\right)^{1/2}\sup_{[0,1]\times[0,1]}(\hat{H_{1}}_{m}-\hat{H_{2}}_{n})(s,t)
\end{equation}

and similarly for the supermodular case:

\begin{equation}\label{gamma}
 \gamma_{m,n}=\left(\frac{mn}{m+n}\right)^{1/2}\sup_{[0,1]\times[0,1]}(\hat{L_{1}}_{m}-\hat{L_{2}}_{n})(s,t)
\end{equation}

Note that this test statistics are multidimensional generalizations of those in \cite{barret1}.

Working with the same definitions and references as in Sec.\ref{SD1} we'll define two operators, $T_{3},T_{4}:\mathit{l}^{\infty}(\mathcal{F})\rightarrow\mathbb{R}$.

To define $T_{3}$, we start by assigning to each element $P\in\mathit{l}^{\infty}(\mathcal{F})$ a function $F_{P}:[0,1]\times[0,1]\rightarrow\mathbb{R}$ given by:

\begin{equation}\label{fp}
F_{P}(s,t)=P1_{[0,s]\times[0,t]} 
\end{equation}

As $P\in\mathit{l}^{\infty}(\mathcal{F})$ the function in \ref{fp} is bounded, and so it turns to be Lebesgue integrable\footnote{Should it be the case that function \ref{fp} was not measurable (perhaps $P$ is some strange element in $\mathit{l}^{\infty}(\mathcal{F})$), we would use the exterior integral. For the case of finite signed measures it has to be masurable, because it is pretty much the same as a distribution function.} over compact set $[0,1]\times[0,1]$. 

We can thus define:

\begin{equation}
T_{3}(P)=\sup_{(x,y)\in[0,1]\times[0,1]}\int_{0}^{x}\int_{0}^{y}F_{P}(s,t)dsdt 
\end{equation}

Note that if $P$ probability measure, then $F_{P}$ turns out to be its associated cdf and we have:

\begin{equation}
T_{3}(P-Q)=\sup_{(x,y)\in[0,1]\times[0,1]}\int_{0}^{x}\int_{0}^{y}(F_{P}-F_{Q})(s,t)dsdt 
\end{equation}

For the supermodular case let's start by defining for each $P\in\mathit{l}^{\infty}(\mathcal{F})$ the function $K_{P}:[0,1]\times[0,1]\rightarrow\mathbb{R}$ by:

\begin{equation}
K_{P}(s,t)=P1_{[0,s]\times[0,1]}+P1_{[0,1]\times[0,t]}-P1_{[0,s]\times[0,t]} 
\end{equation}

and as above we can define:

\begin{equation}
T_{4}(P)=\sup_{(x,y)\in[0,1]\times[0,1]}\int_{0}^{x}\int_{0}^{y}K_{P}(s,t)dsdt 
\end{equation}

so that for the case of probability measures:

\begin{equation}
T_{4}(P-Q)=\sup_{(x,y)\in[0,1]\times[0,1]}\int_{0}^{x}\int_{0}^{y}(K_{P}-K_{Q})(s,t)dsdt 
\end{equation}

We are going to use two sample bootstrap to find critical values for test statistics. The following theorem is analogous to the one we proved for first order bivariate dominance.

\begin{thm}[Second Order Bivariate Stochastic Dominance Test]\label{mioSD2}: Let's consider test statistics $\mu_{m,n}^{B}$ y  $\gamma_{m,n}^{B}$ obtained using the empirical two sample bootstrap distributions (given by \ref{bootempi}) in formulas \ref{mu} and \ref{gamma}. Then if $m,n\rightarrow\infty$ such that $m/N\rightarrow\alpha\in(0,1)$:

\begin{enumerate}
 \item $\mu_{m,n}^{B}\leadsto T_{3}(G_{H})$ 
 \item $\gamma_{m,n}^{B}\leadsto T_{4}(G_{H})$
\end{enumerate}
given almost any sequence $V_{1},V_{2},...Y_{1},Y_{2},...$, where $H=\alpha P+(1-\alpha)Q$ y $G_{H}$ is a version of the tight Brownian bridge corresponding to  measure $H$.
\end{thm}

\begin{proof} We are in the same conditions as in Theorem \ref{mioSD1}, so it will suffice to show that $T_{3}$ and $T_{4}$ are continuous, since $\mu_{m,n}^{B}=T_{3}(\sqrt{\frac{mn}{m+n}}(\hat{P}_{m,N}-\hat{Q}_{n,N}))$ and $\gamma_{m,n}^{B}=T_{4}(\sqrt{\frac{mn}{m+n}}(\hat{P}_{m,N}-\hat{Q}_{n,N}))$.

Now if $P$ and $Q$ are elements in $\mathit{l}^{\infty}(\mathcal{F})$ we have:

\begin{multline}
 |T_{3}(P)-T_{3}(Q)|=\\=\left|\sup_{(x,y)\in[0,1]\times[0,1]}{\int_{0}^{x}\int_{0}^{y}F_{P}(s,t)dsdt}\hspace{0.2cm}-\sup_{(x,y)\in[0,1]\times[0,1]}{\int_{0}^{x}\int_{0}^{y}F_{Q}(s,t)dsdt}\right|\leq\\\leq\sup_{(x,y)\in[0,1]\times[0,1]}{\left|\int_{0}^{x}\int_{0}^{y}F_{P}(s,t)dsdt\hspace{0.2cm}-\int_{0}^{x}\int_{0}^{y}F_{Q}(s,t)dsdt\right|}\leq\\\leq\sup_{(x,y)\in[0,1]\times[0,1]}{\left|\int_{0}^{x}\int_{0}^{y}(F_{P}-F_{Q})(s,t)dsdt\right|}\leq\\\leq\sup_{(x,y)\in[0,1]\times[0,1]}{\int_{0}^{x}\int_{0}^{y}|(F_{P}-F_{Q})(s,t)|dsdt}\leq\\\leq\sup_{(x,y)\in[0,1]\times[0,1]}{\int_{0}^{x}\int_{0}^{y}\parallel P-Q\parallel_{\mathcal{F}}dsdt}\leq\parallel P-Q\parallel_{\mathcal{F}}
\end{multline}

which implies $T_{3}$ is continuous and this proves the first part. 

The same technique can be used to prove the second part, using triangular inequality as we did to prove the second part of Theorem\ref{mioSD1}.
\end{proof}

Once we established the result, we can follow the same reasoning as in Sec.\ref{SD1} to obtain data dependent critical values for the sequence of tests, which turns out to be consistent in the same sense as before. 

\section{Further Generalizations and Discussion of Applications}

There are two obvious ways in which we could want to generalize our results. On the one hand, for higher orders of dominance in the bivariate case, and on the other hand, for the case with more than two variables.

There are not much difficulties in doing so, at least formally. We can as \cite{perez1} define a sequence of operators $T_{j}$ over the space $\mathit{l}^{\infty}(\mathcal{F})$ and proceed as before to show that those operators are continuous. 

For stochastic dominance in higher dimensions we can also generalize modularity classes in a direct way, just by imposing conditions on the sign of higher order partial derivatives. If we can get around the complexity of the great number of integrals that the use of the integration by parts formula would imply even in the three dimensional case, we could obtain sufficient conditions generalizing those in \cite{hadarrussell74}, \cite{atkinson82} and \cite{perez2}.

But the key point here is not of formal but if practical kind. Both in higher orders of dominance and in higher dimensions we are dealing with an increasing number of multivariate integrals. The problem arises then as a manifestation of the well known \textit{curse of dimensionality}: as we increase the number of integrals involved in our tests, the size of the data sets needed for efficient testing grows exponentially. 

As we are dealing with bootstrap re-sampling, this in turn implies that the amount of computations involved rapidly grows. 

So in most practical applications at the time we have to content ourselves with first and second order bivariate stochastic dominance. As shown in \cite{perez1} this suffices for interesting applications in the field of welfare economics, such as the interaction between human capital and income and its impact on two-dimensional economic inequality.

\end{document}